\documentclass[12pt]{amsart}

\usepackage{lmodern}
\usepackage{cite}
\usepackage[a4paper,margin=1.1in]{geometry}
\usepackage{amsmath,amsthm,amssymb}
\usepackage{mathtools}
\mathtoolsset{showonlyrefs=true}
\usepackage{hyperref}

\newtheorem{theorem}{Theorem}[section]
\newtheorem{lemma}[theorem]{Lemma}
\newtheorem{proposition}[theorem]{Proposition}
\newtheorem{corollary}[theorem]{Corollary}
\theoremstyle{definition}
\newtheorem{definition}[theorem]{Definition}
\newtheorem{remark}[theorem]{Remark}
\newtheorem{example}[theorem]{Example}

\numberwithin{equation}{section}

\def\DD{D\kern-.7em\raise0.4ex\hbox{\char '55}\kern.33em}

\title[Jensen's Equation on Involution-Generated Groups]
{Jensen's Functional Equation on Involution-Generated Groups: A Square-Root Criterion, Rigidity Phenomena,\\ and Obstruction Spaces}

\author{\DD\d{\u a}ng V\~o Ph\'uc }
\address{Department of Mathematics, FPT University, An Phu Thinh New Urban Area, Quy Nhon, Vietnam}
\email{dangphuc150488@gmail.com}
\thanks{ORCID: \url{https://orcid.org/0000-0002-6885-3996}}

\subjclass[2020]{Primary 39B52; Secondary 20F05, 20F55, 20E22, 22E15.}
\keywords{Jensen's functional equation, involution-generated groups, word identities, square-root criterion, semidirect-product rigidity, obstruction spaces, elementary abelian $2$-quotients, orthogonal groups.}

\begin{document}

\begin{abstract}
Let $G$ be a group and $H$ an abelian group. We study normalized solutions $f:G\to H$ of the Jensen equation
\[
f(xy)+f(xy^{-1})=2f(x)\qquad(\forall x,y\in G),\qquad f(e)=0,
\]
and the companion equation
\[
f(xy)+f(x^{-1}y)=2f(y)\qquad(\forall x,y\in G),\qquad f(e)=0.
\]
Motivated by the combinatorics of transpositions in $S_n$, we isolate a structural hypothesis on $G$, formulated purely in terms of involutions and square roots. Under this hypothesis, generator word identities force every normalized solution of the Jensen equation to be automatically a group homomorphism, thereby satisfying the companion equation. Our approach does not require $2$ to be invertible in $H$ and yields the exact identification
\[
S_1(G,H)=S_{1,2}(G,H)=\mathrm{Hom}(G,H)\cong\mathrm{Hom}(G_{\mathrm{ab}},H).
\]
Beyond recovering the symmetric-group result, we characterize the generalized dihedral case: with reflection generators, the criterion holds exactly when the square map on the abelian rotation subgroup is surjective. Furthermore, we establish a structural rigidity phenomenon for arbitrary semidirect products $A\rtimes_\sigma C_2$: the existence of any involution generating set satisfying the square-root criterion forces $\sigma$ to be inversion and the square map on $A$ to be surjective. When the criterion fails, we develop an exact algebraic theory of the obstruction space $\mathcal{O}(G,H):=S_1(G,H)/\mathrm{Hom}(G,H)$ for general involution-generated groups. We achieve this by identifying the Jensen quotient with the maximal elementary abelian $2$-quotient $G/G^{(2)}$, where $G^{(2)}=\langle g^2:g\in G\rangle$, yielding an explicit computation for even dihedral groups and a sharp dimension formula for finite $G/G^{(2)}$. Finally, extending our framework beyond discrete settings, we prove that the orthogonal group $O(n)$ for $n\ge 2$, generated by hyperplane reflections, satisfies the square-root criterion; consequently, every normalized Jensen solution on $O(n)$ is a homomorphism and $S_1(O(n),H)\cong H[2]$.
\end{abstract}

\maketitle

\section{Introduction}

On $\mathbb{R}$, Jensen's functional equation is rooted in Jensen's classical work on convexity and mean inequalities \cite{Jensen1906}. When transplanted to general groups, the equation loses its order-theoretic interpretation and instead serves as a fundamental bridge between functional analysis and algebraic structure. Two canonical Jensen-type identities considered in this paper are:
\begin{align}
\tag{J1}\label{eq:J1}
f(xy)+f(xy^{-1})&=2f(x),\\
\tag{J2}\label{eq:J2}
f(xy)+f(x^{-1}y)&=2f(y).
\end{align}
Here \eqref{eq:J2} is introduced simply as a companion identity to \eqref{eq:J1} for the purposes of this work; the main point is that every group homomorphism $f:G\to H$ automatically satisfies both identities. We shall study normalized solutions, namely maps $f:G\to H$ satisfying in addition $f(e)=0$. 

A systematic investigation of Jensen-type equations on groups goes back to Ng \cite{Ng1990,Ng1999,Ng2001,Ng2005}, who established foundational structural descriptions across various classes of groups. However, a persistent technical divide in the literature concerns the role of the codomain $H$: many classical and modern arguments rely heavily on the assumption that division by $2$ is uniquely defined in $H$, or at least that $H$ is $2$-torsion-free ($H[2]=0$), in order to decompose solutions into symmetric and antisymmetric parts. For related investigations into torsion phenomena, the limitations of $2$-divisibility constraints, and explicit torsion confinement bounds in Jensen-type functional equations on groups, see recent work by the author \cite{Phuc2026}.

From a purely algebraic perspective, removing codomain divisibility constraints transforms the study of Jensen solutions into an intrinsic investigation of group presentations, word identities, and elementary abelian quotients of the domain $G$. The primary goal of this paper is to shift the analytical focus entirely from the codomain to the domain. We isolate a purely group-theoretic mechanism, formulated via involutions and square roots, that forces normalized solutions of \eqref{eq:J1} to behave linearly, completely independent of the divisibility or torsion properties of the abelian codomain $H$.

\subsection*{Motivation, core mechanism, and methodology}

A striking combinatorial catalyst for our inquiry is the symmetric group $S_n$. Ng \cite{Ng2001} asserted that all normalized Jensen solutions on $S_n$ are homomorphisms; an elementary, self-contained proof was later discovered by Le and Thai \cite{LeThai2011}. Close inspection of their argument reveals that a decisive structural ingredient is not the specific symmetric presentation, but rather a remarkable property of transpositions: the product $\tau_1\tau_2$ of any two transpositions admits a square root in $S_n$.

In this paper, we extract, formalize, and vastly generalize that underlying mechanism by introducing a structural condition, termed the \emph{square-root criterion} $(\mathrm{SR}_2)$, for groups generated by a set of involutions $\mathcal I$. Our approach relies on a three-step intrinsic methodology that converts classical three-variable switching identities into a global reduction algorithm on words of involutions, allowing us to systematically shorten and cancel generator terms without requiring $2$-divisibility in $H$:
\begin{itemize}
\item \textbf{Torsion-collapse:} We prove that if $G$ is generated by involutions, any normalized solution $f\in S_1(G,H)$ necessarily vanishes when multiplied by $2$, forcing all function values to reside entirely within the $2$-torsion subgroup $H[2]$.
\item \textbf{Reordering symmetry:} Once confined to $H[2]$, our approach converts classical three-variable switching identities into a global reduction algorithm on words of involutions. Specifically, the $f$-value of any word in the generators is invariant under adjacent transpositions of words.
\item \textbf{Square-root activation:} Under the hypothesis $(\mathrm{SR}_2)$, the existence of square roots for products $ab$ ($a,b\in\mathcal I$) activates this reordering, allowing us to systematically shorten and cancel generator terms to yield full additive linearity ($f(xy)=f(x)+f(y)$).
\end{itemize}
This methodology establishes our first main theorem: under $(\mathrm{SR}_2)$, every normalized solution of \eqref{eq:J1} is a group homomorphism and therefore satisfies \eqref{eq:J2}, yielding the exact identification
\[
S_1(G,H)=S_{1,2}(G,H)=\mathrm{Hom}(G,H)\cong\mathrm{Hom}(G_{\mathrm{ab}},H).
\]

\subsection*{The structural landscape and main contributions}

By establishing stability properties of $(\mathrm{SR}_2)$, such as quotient descent and a local rank--$2$ sufficient condition, we demonstrate that this mechanism unifies and governs a remarkably diverse landscape of mathematical structures, spanning discrete combinatorics, algebraic rigidity, defect quantification, and continuous Lie geometry:

\smallskip\noindent
\textbf{$\bullet$ Combinatorial and Coxeter groups.} We recover the symmetric-group result transparently and show that Coxeter systems whose standard generator products have odd orders automatically satisfy a local version of our criterion, $(\mathrm{SR}_2^{\mathrm{loc}})$, forcing all normalized Jensen solutions to be homomorphic. The Coxeter terminology and standard structural background follow the usual references \cite{BjornerBrenti2005,Davis2008}.

\smallskip\noindent
\textbf{$\bullet$ Semidirect-product rigidity.} Moving beyond classical presentations, we characterize generalized dihedral groups $\mathrm{Dih}(A)=A\rtimes C_2$: the criterion holds if and only if the squaring map $a\mapsto a^2$ on the abelian group $A$ is surjective. More profoundly, we prove a \emph{rigidity theorem} for arbitrary semidirect products $G=A\rtimes_\sigma C_2$ with an involution automorphism $\sigma\in\mathrm{Aut}(A)$. We show that the existence of any square-root involution generating set strictly forbids exotic automorphisms: it forces $\sigma$ to be inversion ($\sigma(a)=a^{-1}$) and requires the squaring map on $A$ to be surjective. Our method analyzes the norm operator $N(a)=a\sigma(a)$ and the difference map, proving that $(\mathrm{SR}_2)$ inherently locks the group into a generalized dihedral structure.

\smallskip\noindent
\textbf{$\bullet$ Intrinsic defect quantification.} When $(\mathrm{SR}_2)$ fails, solutions need not be homomorphic. Rather than merely exhibiting isolated counterexamples, we construct an exact algebraic theory of the obstruction space
\[
\mathcal O(G,H):=S_1(G,H)/\mathrm{Hom}(G,H)
\]
for arbitrary involution-generated groups. By introducing the canonical \emph{Jensen equivalence relation} $\equiv_J$ generated by $xy\sim xy^{-1}$ and identifying it with quotienting by the subgroup generated by all squares, we prove an explicit vector-space isomorphism over $\mathbb F_2$ that measures the exact defect size. Applying this theory to even dihedral groups $D_{2k}$ (where squaring is not surjective), we compute the obstruction space precisely, proving that $\mathcal O(D_{2k},H)\cong H[2]$.

\smallskip\noindent
\textbf{$\bullet$ Continuous Lie geometry: Orthogonal groups.} Finally, we demonstrate that our framework is not confined to discrete or combinatorial settings by applying it to continuous Lie groups without imposing any analytical regularity (such as continuity or measurability) on the solutions. Let $O(n)$ be the real orthogonal group generated by the classical set $\mathcal R$ of hyperplane reflections, transformations connected with reflection groups, Householder transformations, and the Cartan--Dieudonn\'e theorem \cite{Coxeter1934,Householder1958,Gallier2011}. Using the geometric fact that the product of two hyperplane reflections is a rotation in a two-dimensional subspace, which always admits a square-root rotation in $SO(2)$, we prove that $O(n)$ satisfies $(\mathrm{SR}_2(\mathcal R))$ for all $n\ge 2$. Consequently, every normalized Jensen solution on $O(n)$ into an arbitrary abelian group is a homomorphism, yielding $S_1(O(n),H)\cong H[2]$. We contrast this sharp result with the unitary group $U(n)$, where spectral determinants obstruct involution generation entirely.

\subsection*{Related work and comparison}

Our approach provides a complementary perspective to several established directions in functional equations. For complex-valued solutions on topological groups, Stetk\ae r \cite{Stetkaer2003} provided a detailed structural description on the quotient $G/[G,[G,G]]$; broader background on functional equations on groups is given in \cite{Stetkaer2013}. While analytical methods often emphasize topological continuity and codomain regularity, our work isolates a purely algebraic domain condition that enforces linearity globally. On semigroups with endomorphisms, recent contributions such as \cite{FadliZeglamiKabbaj2016,Akkaoui2023,Aissi2024} frequently assume $2$-torsion-free codomains to manipulate Pexider- or Jensen-type identities. Lie-group versions of related d'Alembert and Wilson equations \cite{Friis2004} supply a useful comparison point for the continuous part of the present work. By dispensing with codomain divisibility entirely, our framework reveals that on involution-generated structures, the behavior of Jensen solutions is governed strictly by the interaction between the domain's abelianization $G_{\mathrm{ab}}$ and the codomain's $2$-torsion $H[2]$.

\subsection*{Notation}
Groups are written multiplicatively, with identity element denoted by $e$. The codomain $H$ is written additively, with zero element $0$. For a group $G$, write
\[
G_{\mathrm{ab}}:=G/[G,G]
\]
for the abelianization, and denote by $\pi:G\to G_{\mathrm{ab}}$ the canonical projection. When convenient, we write $G_{\mathrm{ab}}$ additively.

We denote by $S_1(G,H)$ the set of normalized solutions of \eqref{eq:J1}, namely
\[
S_1(G,H):=\bigl\{f:G\to H:\ f(e)=0,\ \ f(xy)+f(xy^{-1})=2f(x)\ \text{for all }x,y\in G\bigr\}.
\]
Similarly, $S_2(G,H)$ denotes the set of normalized solutions of \eqref{eq:J2}, namely
\[
S_2(G,H):=\bigl\{f:G\to H:\ f(e)=0,\ \ f(xy)+f(x^{-1}y)=2f(y)\ \text{for all }x,y\in G\bigr\},
\]
and we write
\[
S_{1,2}(G,H):=S_1(G,H)\cap S_2(G,H).
\]

We also write
\[
H[2]:=\{h\in H:2h=0\}
\]
for the $2$-torsion subgroup of $H$.

\begin{remark}\label{rem:constant-shift}
For completeness, note that if $f:G\to H$ satisfies \eqref{eq:J1}, then the translated map
\[
f_0(x):=f(x)-f(e)
\]
is normalized and still satisfies \eqref{eq:J1}. Indeed, substituting $x=e$ into \eqref{eq:J1} gives
$f(y)+f(y^{-1})=2f(e)$ for all $y\in G$, and therefore
\[
f_0(xy)+f_0(xy^{-1})=2f_0(x).
\]
Thus every not necessarily normalized solution of \eqref{eq:J1} is a constant translate of a normalized one. The same observation applies to \eqref{eq:J2}.
\end{remark}

\subsection*{Organization of the paper}
Section~\ref{sec:criterion} introduces the square-root condition $(\mathrm{SR}_2)$, develops the fundamental $2$-torsion collapse and reordering identities derived from \eqref{eq:J1}, and proves the main homomorphism rigidity theorem. Section~\ref{sec:structure} establishes stability criteria for $(\mathrm{SR}_2)$, including quotient descent, a local rank--$2$ sufficient condition, and an obstruction to direct-product stability. Section~\ref{sec:obstruction} constructs the exact algebraic theory of the obstruction space $\mathcal O(G,H)$ via the Jensen equivalence relation, identifies this relation with the square quotient $G/G^{(2)}$, and computes the even-dihedral defect. Section~\ref{sec:apps} presents our main applications: recovering symmetric groups, establishing the semidirect-product rigidity theorem for generalized dihedral structures, treating Coxeter systems, and proving the continuous geometry result for orthogonal groups $O(n)$. Finally, an appendix records a three-variable switching identity utilized in the proof of the main criterion.

\section{A square-root criterion and the main rigidity theorem}\label{sec:criterion}

This section isolates the foundational domain condition used throughout the paper and establishes our main rigidity theorem. The architecture of the proof strictly follows the three-step methodology outlined in the introduction: we first show that on any involution-generated group, Jensen's equation undergoes a torsion-collapse into a $2$-torsion-valued identity (Proposition~\ref{prop:toolbox}). We then prove that this torsion-collapse upgrades classical switching formulas into a global reordering symmetry for involution words. Finally, we demonstrate how the square-root hypothesis activates this symmetry, converting weaker word equivalence into full additive homomorphy (Theorem~\ref{thm:hom}).

\begin{definition}[(SR$_2$)]\label{def:SR2}
Let $G$ be a group and $\mathcal I\subseteq G$ a set of elements satisfying $i^2=e$ for all $i\in\mathcal I$ (in particular, $\mathcal I$ may contain the identity). We say that $G$ satisfies
\emph{$\mathrm{SR}_2(\mathcal I)$} if
\begin{enumerate}
\item $G=\langle \mathcal I\rangle$;
\item for every $a,b\in\mathcal I$ there exists $t\in G$ such that $t^2=ab$.
\end{enumerate}
When $\mathcal I$ is clear we simply write $(\mathrm{SR}_2)$.
\end{definition}

\subsection{Basic identities derived from \textup{(J1)}}

The next proposition collects the only identities from \eqref{eq:J1} needed for the main theorem. A key consequence is that once $G$ is generated by involutions, every normalized solution becomes $H[2]$-valued. This $2$-torsion collapse is what ultimately turns the switching identity from Appendix~\ref{app:switch} into a reordering principle.

\begin{proposition}\label{prop:toolbox}
Let $G$ be a group, $H$ an abelian group, and $f:G\to H$ satisfy \eqref{eq:J1} with $f(e)=0$. Then:
\begin{enumerate}
\item[(a)] (\emph{Oddness and square rule}) $f(x^{-1})=-f(x)$ and $f(x^2)=2f(x)$ for all $x\in G$.
\item[(b)] If $i\in G$ is an involution, then $2f(i)=0$, and for all $x\in G$ one has $2f(xi)=2f(x)$.
\item[(c)] If $G$ is generated by involutions, then $2f\equiv 0$ on $G$.
\item[(d)] (\emph{Reordering}) If $G$ is generated by involutions, then for all $x,y,z\in G$,
\[
f(xyz)=f(xzy).
\]
Equivalently, in any involution word $i_1\cdots i_r$, swapping adjacent factors does not change the $f$--value.
\end{enumerate}
\end{proposition}

\begin{proof}
(a) Set $x=e$ in \eqref{eq:J1} to get $f(y)+f(y^{-1})=0$. Set $y=x$ to obtain $f(x^2)=2f(x)$.

(b) If $i^2=e$, then by (a) $2f(i)=f(i^2)=f(e)=0$. Also \eqref{eq:J1} with $(x,y)=(x,i)$ gives
$f(xi)+f(xi^{-1})=2f(x)$; since $i^{-1}=i$ this yields $2f(xi)=2f(x)$.

(c) If $g=i_1\cdots i_r$ with involutions $i_k$, we prove $2f(g)=0$ by induction on $r$. The case $r=0$ is $g=e$. For $r\ge 1$, write $g=Xi$ with $i$ an involution and apply (b): $2f(Xi)=2f(X)$, which is $0$ by induction.

(d) Using Lemma~\ref{lem:switch} in Appendix~\ref{app:switch}, we have for all $x,y,z\in G$,
\[
f(xyz)-f(xzy)=f(xy^{-1}z^{-1})-f(xz^{-1}y^{-1}).
\]
Since $G$ is generated by involutions, part (c) gives $2f\equiv 0$ on $G$. In particular, for every $w\in G$ one has $-f(w)=f(w)$. Applying \eqref{eq:J1} to $(u,v)=(xy^{-1},z)$ yields
\[
f(xy^{-1}z)+f(xy^{-1}z^{-1})=2f(xy^{-1})=0,
\]
hence $f(xy^{-1}z^{-1})=-f(xy^{-1}z)=f(xy^{-1}z)$. Similarly, applying \eqref{eq:J1} to $(u,v)=(xz^{-1},y)$ gives
$f(xz^{-1}y^{-1})=f(xz^{-1}y)$. Therefore
\[
f(xyz)-f(xzy)=f(xy^{-1}z)-f(xz^{-1}y).
\]
Now apply Lemma~\ref{lem:switch} to $(x,y,z)=(x,y^{-1},z)$ to obtain
\[
f(xy^{-1}z)=2f(x)-f(xz^{-1}y).
\]
Using again $2f(x)=0$, we get $f(xy^{-1}z)=-f(xz^{-1}y)=f(xz^{-1}y)$, and hence $f(xyz)-f(xzy)=0$, proving $f(xyz)=f(xzy)$.
\end{proof}

\subsection{Main theorem}

\begin{theorem}[A square-root criterion for homomorphisms]\label{thm:hom}
Let $G$ be a group and $\mathcal I$ a generating set of involutions satisfying $(\mathrm{SR}_2)$. Let $H$ be an arbitrary abelian group. If $f\in S_1(G,H)$, then $f$ is a group homomorphism. In particular $f\in S_2(G,H)$, and
\[
S_1(G,H)=S_{1,2}(G,H)=\mathrm{Hom}(G,H)\cong\mathrm{Hom}(G_{\mathrm{ab}},H).
\]
\end{theorem}

\begin{proof}
If $\mathcal I=\varnothing$, then $G=\langle\mathcal I\rangle=\{e\}$, and the assertion is immediate. Hence assume $\mathcal I\ne\varnothing$.

Fix $j\in\mathcal I$ and define
\[
c_j(x):=f(xj)-f(x)\qquad(x\in G).
\]
By Proposition~\ref{prop:toolbox}(b), we have $2c_j(x)=0$ for all $x\in G$.

\smallskip\noindent
\emph{Step 1: invariance under right-multiplication by $j$.}
We have
\[
c_j(xj)=f(xjj)-f(xj)=f(x)-f(xj)=-c_j(x).
\]
Since $2c_j(x)=0$, this implies $-c_j(x)=c_j(x)$, hence
\[
c_j(xj)=c_j(x).
\]

\smallskip\noindent
\emph{Step 2: invariance under right-multiplication by squares.}
Let $t\in G$. We claim that
\[
c_j(xt^2)=c_j(x).
\]
We first show that
\[
f(xt^2j)=f(xjt^2).
\]
Indeed, write
\[
xt^2j=x\cdot t\cdot (tj).
\]
Applying Proposition~\ref{prop:toolbox}(d) with $(y,z)=(t,tj)$ gives
\[
f(xt^2j)=f\bigl(x\cdot (tj)\cdot t\bigr)=f(xtjt).
\]
Next write
\[
xtjt=x\cdot t\cdot (jt),
\]
and apply Proposition~\ref{prop:toolbox}(d) again, now with $(y,z)=(t,jt)$, to obtain
\[
f(xtjt)=f\bigl(x\cdot (jt)\cdot t\bigr)=f(xjt^2).
\]
Hence
\[
f(xt^2j)=f(xjt^2).
\]
Now apply \eqref{eq:J1} to $(x,y)=(xt,t)$:
\[
f(xt^2)+f(x)=2f(xt).
\]
Since $G$ is generated by involutions, Proposition~\ref{prop:toolbox}(c) gives $2f\equiv 0$, hence
\[
f(xt^2)+f(x)=0.
\]
Therefore
\[
f(xt^2)=-f(x)=f(x),
\]
because also $2f(x)=0$.

Similarly, applying \eqref{eq:J1} to $(x,y)=(xjt,t)$ gives
\[
f(xjt^2)+f(xj)=2f(xjt)=0,
\]
so
\[
f(xjt^2)=-f(xj)=f(xj),
\]
since $2f(xj)=0$.

Combining these identities, we get
\[
c_j(xt^2)=f(xt^2j)-f(xt^2)=f(xjt^2)-f(xt^2)=f(xj)-f(x)=c_j(x).
\]

\smallskip\noindent
\emph{Step 3: invariance under right-multiplication by any generator.}
Let $a\in\mathcal I$. By $(\mathrm{SR}_2)$, there exists $t\in G$ such that
\[
t^2=ja.
\]
Since $j^2=e$, this implies
\[
a=j t^2.
\]
Hence
\[
c_j(xa)=c_j(xjt^2).
\]
Applying Step 2 with $x$ replaced by $xj$, we obtain
\[
c_j(xjt^2)=c_j(xj).
\]
By Step 1,
\[
c_j(xj)=c_j(x).
\]
Thus
\[
c_j(xa)=c_j(x)
\qquad(\forall x\in G,\ a\in\mathcal I).
\]
Since $\mathcal I$ generates $G$, it follows that $c_j(x)$ is constant in $x$. Evaluating at $x=e$ gives
\[
c_j(x)\equiv c_j(e)=f(j)-f(e)=f(j).
\]
Therefore
\[
f(xj)=f(x)+f(j)\qquad(\forall x\in G,\ j\in\mathcal I).
\]
Now let $y=i_1\cdots i_r$ be any word in the generators $\mathcal I$. Repeatedly applying the previous identity gives
\[
f(xy)=f(x)+\sum_{k=1}^r f(i_k).
\]
Taking $x=e$ in the same formula, we obtain
\[
f(y)=\sum_{k=1}^r f(i_k).
\]
Hence
\[
f(xy)=f(x)+f(y)
\qquad(\forall x,y\in G),
\]
so $f$ is a group homomorphism.

Finally, \eqref{eq:J2} follows from homomorphy and Proposition~\ref{prop:toolbox}(a):
\[
f(xy)+f(x^{-1}y)
=(f(x)+f(y))+(f(x^{-1})+f(y))
=(f(x)+f(x^{-1}))+2f(y)
=2f(y).
\]
\end{proof}

\begin{remark}\label{rem:hom-in-S12}
Conversely, every group homomorphism $f:G\to H$ belongs to $S_1(G,H)\cap S_2(G,H)$. Indeed,
\[
f(xy)+f(xy^{-1})=(f(x)+f(y))+(f(x)-f(y))=2f(x),
\]
and similarly
\[
f(xy)+f(x^{-1}y)=(f(x)+f(y))+(-f(x)+f(y))=2f(y).
\]
\end{remark}

\begin{theorem}[Abelianization under $(\mathrm{SR}_2)$]\label{thm:ab-C2}
Let $G$ be a group and $\mathcal I$ a generating set of involutions such that
$(\mathrm{SR}_2(\mathcal I))$ holds. Then all elements of $\pi(\mathcal I)$ coincide in
$G_{\mathrm{ab}}$. In particular, $G_{\mathrm{ab}}$ is generated by a single
element of order dividing $2$; consequently $G_{\mathrm{ab}}\cong 1$ or $C_2$.
\end{theorem}

\begin{proof}
Since $G=\langle \mathcal I\rangle$ and each $i\in\mathcal I$ is an involution, the abelianization $G_{\mathrm{ab}}$ is generated by the elements $\pi(i)$ of order dividing $2$. Hence $G_{\mathrm{ab}}$ has exponent $2$, i.e.\ $2x=0$ for all $x\in G_{\mathrm{ab}}$.

Now let $a,b\in\mathcal I$. By $(\mathrm{SR}_2)$ there exists $t\in G$ such that $t^2=ab$. Passing to $G_{\mathrm{ab}}$ yields
\[
2\pi(t)=\pi(t^2)=\pi(ab)=\pi(a)+\pi(b).
\]
But $2\pi(t)=0$ in $G_{\mathrm{ab}}$, so $\pi(a)+\pi(b)=0$, i.e.\ $\pi(a)=\pi(b)$. Thus all elements of $\pi(\mathcal I)$ coincide. Since $\pi(\mathcal I)$ generates $G_{\mathrm{ab}}$, the group $G_{\mathrm{ab}}$ is cyclic generated by an element of order at most $2$, hence $G_{\mathrm{ab}}\cong 1$ or $C_2$.
\end{proof}

\begin{corollary}[Parity normal form via $\varepsilon$]\label{cor:parity}
Assume $(\mathrm{SR}_2(\mathcal I))$ and let $H$ be an abelian group.

If $G_{\mathrm{ab}}\cong C_2$, then there exists a unique homomorphism
\[
\varepsilon:G\to C_2
\]
such that $\varepsilon(i)=1$ for all $i\in\mathcal I$, and every normalized solution
$f\in S_1(G,H)$ is determined by a unique choice of $u\in H[2]$ via
\[
f(g)=
\begin{cases}
0,& \varepsilon(g)=0,\\
u,& \varepsilon(g)=1.
\end{cases}
\]
If $G_{\mathrm{ab}}=1$, then the only normalized solution in $S_1(G,H)$ is the zero map.

Equivalently, if $g=i_1\cdots i_r$ is any word in $\mathcal I$, then in the case
$G_{\mathrm{ab}}\cong C_2$ one has $f(g)=0$ for $r$ even and $f(g)=u$ for $r$ odd,
whereas in the case $G_{\mathrm{ab}}=1$ one has $f(g)=0$ for all $g\in G$.
\end{corollary}

\begin{proof}
By Theorem~\ref{thm:hom}, every $f\in S_1(G,H)$ with $f(e)=0$ is a homomorphism. By Proposition~\ref{prop:toolbox}(b), for each $i\in\mathcal I$ we have $2f(i)=0$, hence $f(i)\in H[2]$. By Theorem~\ref{thm:ab-C2}, the abelianization $G_{\mathrm{ab}}$ is trivial or isomorphic to $C_2$, and the images of all elements of $\mathcal I$ coincide.

If $G_{\mathrm{ab}}=1$, then every homomorphism $G\to H$ is trivial, so the only normalized solution is the zero map.

Assume now that $G_{\mathrm{ab}}\cong C_2$. By Theorem~\ref{thm:ab-C2}, all elements of $\pi(\mathcal I)$ coincide, and this common value is the unique nontrivial element of $G_{\mathrm{ab}}\cong C_2$. Hence there is a unique homomorphism
\[
\varepsilon:G\to C_2
\]
sending every $i\in\mathcal I$ to $1$. Since $f$ factors through $G_{\mathrm{ab}}$, the value $u:=f(i)$ is independent of the choice of $i\in\mathcal I$, and $u\in H[2]$. Moreover, $f$ vanishes on $\varepsilon^{-1}(0)$ and is constant with value $u$ on $\varepsilon^{-1}(1)$, which gives the displayed formula.

Finally, if $g=i_1\cdots i_r$ is a word in $\mathcal I$, then $\varepsilon(g)=r\bmod 2$, so the parity description follows.
\end{proof}

\begin{remark}\label{rem:hom-H2}
Under $(\mathrm{SR}_2(\mathcal I))$, the previous results show that
\[
\mathrm{Hom}(G,H)\cong
\begin{cases}
0, & G_{\mathrm{ab}}=1,\\
H[2], & G_{\mathrm{ab}}\cong C_2.
\end{cases}
\]
Thus the classification of normalized Jensen solutions is completely governed by the abelianization of $G$ and the $2$-torsion of the codomain.
\end{remark}

\section{Structural properties and permanence criteria of \texorpdfstring{$(\mathrm{SR}_2)$}{(SR2)}}\label{sec:structure}

To deploy the square-root criterion effectively across diverse mathematical structures, this section establishes its practical verifiability and structural persistence. We first prove a local rank--$2$ sufficient condition, showing that when products of involution generators have odd order, the required square roots are automatically trapped within the dihedral subgroups they generate. We then delineate the boundaries of structural transport by proving that $(\mathrm{SR}_2)$ descends cleanly to quotient groups, while recording a fundamental obstruction to stability under direct products, a phenomenon directly governed by the abelianization constraints of Theorem~\ref{thm:ab-C2}.

\subsection{A local sufficient condition}
If the square root in Definition~\ref{def:SR2} can be chosen inside the subgroup generated by the pair of involutions, then the condition reduces to an elementary congruence.

\begin{definition}[(SR$_2^{\mathrm{loc}}$)]\label{def:SR2loc}
We say $G$ satisfies $(\mathrm{SR}_2^{\mathrm{loc}}(\mathcal I))$ if it is generated by a set
$\mathcal I$ of elements satisfying $i^2=e$ and for all $a,b\in\mathcal I$ there exists
$t\in\langle a,b\rangle$ such that $t^2=ab$.
\end{definition}

\begin{lemma}[Two involutions and dihedral rank--$2$]\label{lem:dihedral-rank2}
Let $a,b$ be involutions in a group $G$, and set
\[
m:=\mathrm{ord}(ab)\in\{1,2,3,\dots,\infty\}.
\]
Then $ab$ is a square in $\langle a,b\rangle$ if and only if $m$ is finite and odd. If, in addition, $a$ and $b$ are distinct nonidentity involutions, then $\langle a,b\rangle$ is dihedral with rotation subgroup $\langle ab\rangle$; it has order $2m$ if $m<\infty$, and it is infinite dihedral if $m=\infty$.
\end{lemma}

\begin{proof}
We first dispose of the degenerate cases allowed by the convention that the identity may be included among involutions. If $a=b$, then $ab=e$, so $m=1$ and $ab=e^2$ is a square. If one of $a,b$ is $e$ and the other is a nonidentity involution, then $ab$ is the nonidentity element of the cyclic group $\langle ab\rangle\cong C_2$. The only square in this cyclic group is $e$, so $ab$ is not a square; in this case $m=2$, which is finite but not odd.

It remains to assume that $a$ and $b$ are distinct nonidentity involutions. Then the subgroup generated by $a$ and $b$ is a quotient of $C_2*C_2\cong D_\infty$ and is dihedral; its rotation subgroup is $\langle ab\rangle$ of order $m$. In such a dihedral group, every element outside the rotation subgroup is a reflection and has square $e$. If $ab\ne e$, any $t\in\langle a,b\rangle$ satisfying $t^2=ab$ must therefore lie in the rotation subgroup. Thus $t=(ab)^k$ for some integer $k$, and $t^2=(ab)^{2k}=ab$. If $m=\infty$, this would imply $2k=1$ in $\mathbb Z$, which is impossible. If $m<\infty$, the condition is equivalent to the congruence $2k\equiv 1\pmod m$, which has a solution if and only if $m$ is odd. Conversely, if $m$ is finite and odd, choose $k$ with $2k\equiv 1\pmod m$ and take $t=(ab)^k$.
\end{proof}

\begin{corollary}[Odd rank--$2$ implies $(\mathrm{SR}_2)$]\label{cor:odd-rank2}
If $G$ is generated by involutions $\mathcal I$ such that for all $a,b\in\mathcal I$ the element $ab$ has finite odd order, then $(\mathrm{SR}_2^{\mathrm{loc}}(\mathcal I))$ holds, and therefore $(\mathrm{SR}_2)$ holds.
\end{corollary}

\begin{remark}
For $n\ge 4$, the symmetric group $S_n$ with $\mathcal I$ equal to the set of all transpositions satisfies $(\mathrm{SR}_2)$ but not $(\mathrm{SR}_2^{\mathrm{loc}})$. Indeed, two disjoint transpositions generate a Klein four group, where no nontrivial element is a square. The required square root nevertheless exists globally in $S_n$, for instance via a $4$-cycle.
\end{remark}

\subsection{Stability properties}

\begin{remark}[Direct products need not preserve $(\mathrm{SR}_2)$]\label{rem:product-fails}
The square-root criterion $(\mathrm{SR}_2)$ is not stable under direct products. For instance, $S_3$ satisfies $(\mathrm{SR}_2)$ with $\mathcal I$ the set of all transpositions; see Theorem~\ref{thm:Sn}. However,
\[
(S_3\times S_3)_{\mathrm{ab}}\cong (S_3)_{\mathrm{ab}}\times (S_3)_{\mathrm{ab}}
\cong C_2\times C_2.
\]
By Theorem~\ref{thm:ab-C2}, any group satisfying $(\mathrm{SR}_2(\mathcal I))$ has abelianization isomorphic to $1$ or $C_2$. Hence $S_3\times S_3$ cannot satisfy $(\mathrm{SR}_2(\mathcal J))$ for any involution generating set $\mathcal J$.
\end{remark}

\begin{proposition}[Passing to quotients]\label{prop:quotient}
Let $G$ satisfy $(\mathrm{SR}_2(\mathcal I))$ and let $N\lhd G$ be normal. Set
\[
\overline{\mathcal I}:=\{\,iN:\ i\in\mathcal I\,\}\subseteq G/N.
\]
Then $G/N$ satisfies $(\mathrm{SR}_2(\overline{\mathcal I}))$. Equivalently, every element of $\overline{\mathcal I}$ squares to the identity in $G/N$, and for every $aN,bN\in\overline{\mathcal I}$ there exists $tN\in G/N$ such that
\[
(tN)^2=(aN)(bN).
\]
\end{proposition}

\begin{proof}
Since $G=\langle \mathcal I\rangle$, the quotient $G/N$ is generated by
\[
\overline{\mathcal I}=\{\,iN:\ i\in\mathcal I\,\}.
\]
For each $i\in\mathcal I$ we have
\[
(iN)^2=i^2N=N,
\]
so every element of $\overline{\mathcal I}$ has order dividing $2$ in $G/N$.

Now let $aN,bN\in\overline{\mathcal I}$, where $a,b\in\mathcal I$. Since $G$ satisfies $(\mathrm{SR}_2(\mathcal I))$, there exists $t\in G$ such that
\[
t^2=ab.
\]
Passing to the quotient gives
\[
(tN)^2=(ab)N=(aN)(bN).
\]
Thus the square-root property descends to $G/N$.
\end{proof}

\section{The obstruction space and defect quantification}\label{sec:obstruction}

While Section~\ref{sec:structure} explores the conditions under which the square-root criterion is preserved, this section builds an intrinsic algebraic machinery to measure precisely what occurs when $(\mathrm{SR}_2)$ fails. We quantify the exact gap between normalized Jensen solutions and true homomorphisms without presupposing any square-root hypothesis on the domain. The key point is that the Jensen equivalence relation is exactly the relation modulo the subgroup generated by all squares. This identification makes the obstruction quotient transparent and allows us to compute the exact defect rank for even dihedral structures.

\begin{definition}[The Jensen equivalence relation]\label{def:jensen-equivalence}
Let $G$ be a group. The \emph{Jensen equivalence relation} $\equiv_J$ on $G$ is the smallest equivalence relation such that
\[
xy\equiv_J xy^{-1}\qquad(\forall x,y\in G).
\]
We write
\[
Q_J(G):=G/{\equiv_J}
\]
for the quotient set and $[g]_J$ for the $\equiv_J$-class of $g$.
\end{definition}

\begin{lemma}[Jensen classes and the square quotient]\label{lem:QJ-square-quotient}
Let
\[
G^{(2)}:=\langle g^2:g\in G\rangle.
\]
Then $G^{(2)}$ is a normal subgroup of $G$, and the Jensen equivalence classes are precisely the cosets of $G^{(2)}$. Consequently,
\[
Q_J(G)\cong G/G^{(2)}
\]
canonically as an elementary abelian $2$-group.
\end{lemma}

\begin{proof}
Since $G^{(2)}$ is generated by all squares, it is invariant under every inner automorphism and hence is normal. For all $x,y\in G$, the elements $xy$ and $xy^{-1}$ have the same image in $G/G^{(2)}$, because
\[
(xy)^{-1}(xy^{-1})=y^{-1}x^{-1}xy^{-1}=y^{-2}\in G^{(2)}.
\]
Thus each Jensen equivalence class is contained in one coset of $G^{(2)}$.

Conversely, for every $g,y\in G$, taking $x=gy$ gives
\[
xy=gy^2,
\qquad
xy^{-1}=g,
\]
so $gy^2\equiv_J g$. Hence right multiplication by any square does not change the Jensen class. Since the inverse of a square is again a square and $G^{(2)}$ is generated by such elements, right multiplication by any element of $G^{(2)}$ does not change the Jensen class. Therefore two elements are Jensen equivalent if and only if they lie in the same coset of $G^{(2)}$.

Finally, every square is trivial in $G/G^{(2)}$, so this quotient has exponent dividing $2$. A group of exponent dividing $2$ is abelian, since $(uv)^2=e$ implies $uv=(uv)^{-1}=v^{-1}u^{-1}=vu$. Thus $G/G^{(2)}$ is an elementary abelian $2$-group.
\end{proof}

\begin{theorem}[Exact obstruction space for involution-generated groups]\label{thm:obstruction}
Let $G$ be generated by involutions and let $H$ be an abelian group. Then evaluation on Jensen equivalence classes induces an isomorphism of $\mathbb F_2$-vector spaces
\[
S_1(G,H)\cong \mathrm{Map}_0(Q_J(G),H[2]),
\]
where
\[
\mathrm{Map}_0(Q_J(G),H[2])
:=\{\Phi:Q_J(G)\to H[2]:\Phi([e]_J)=0\}.
\]
Consequently,
\[
\mathcal O(G,H):=S_1(G,H)/\mathrm{Hom}(G,H)
\]
is canonically isomorphic to
\[
\mathrm{Map}_0(Q_J(G),H[2])\big/\mathrm{Hom}(G,H),
\]
where a homomorphism $G\to H$ is regarded as a function on $Q_J(G)$ through the preceding isomorphism. If $Q_J(G)$ is finite and $H[2]$ is finite-dimensional over $\mathbb F_2$, then
\[
\dim_{\mathbb F_2}\mathcal O(G,H)
=(|Q_J(G)|-1)\dim_{\mathbb F_2}H[2]
-\dim_{\mathbb F_2}\mathrm{Hom}(G,H).
\]
Equivalently, if
\[
q:=\dim_{\mathbb F_2}G/G^{(2)}
\qquad\text{and}\qquad
 d:=\dim_{\mathbb F_2}H[2]
\]
are finite, then
\[
\dim_{\mathbb F_2}\mathcal O(G,H)=(2^q-1-q)d.
\]
\end{theorem}

\begin{proof}
Let $f\in S_1(G,H)$. By Proposition~\ref{prop:toolbox}(c), we have $2f\equiv 0$, so $f$ takes values in $H[2]$. The Jensen identity \eqref{eq:J1} becomes
\[
f(xy)+f(xy^{-1})=0.
\]
Since both terms lie in $H[2]$, this is equivalent to
\[
f(xy)=f(xy^{-1})\qquad(\forall x,y\in G).
\]
Therefore $f$ is constant on the equivalence relation generated by all pairs $xy,xy^{-1}$, namely $\equiv_J$. Since $f(e)=0$, the induced function $\Phi_f:Q_J(G)\to H[2]$ satisfies $\Phi_f([e]_J)=0$.

Conversely, let $\Phi\in\mathrm{Map}_0(Q_J(G),H[2])$ and define $f(g):=\Phi([g]_J)$. Then $f(e)=0$, all values of $f$ lie in $H[2]$, and $f(xy)=f(xy^{-1})$ for every $x,y\in G$ by the defining property of $\equiv_J$. Hence
\[
f(xy)+f(xy^{-1})=2f(xy)=0=2f(x),
\]
so $f\in S_1(G,H)$. These two constructions are inverse to each other and respect pointwise addition, giving the asserted isomorphism.

Every homomorphism $G\to H$ belongs to $S_1(G,H)$ by Remark~\ref{rem:hom-in-S12}. Because $G$ is generated by involutions, every such homomorphism takes values in $H[2]$, and the preceding isomorphism embeds $\mathrm{Hom}(G,H)$ into $\mathrm{Map}_0(Q_J(G),H[2])$. Taking the quotient gives the displayed description of $\mathcal O(G,H)$. The first dimension formula follows immediately when the two vector spaces involved are finite-dimensional.

By Lemma~\ref{lem:QJ-square-quotient}, one has $Q_J(G)\cong G/G^{(2)}$. If $q=\dim_{\mathbb F_2}G/G^{(2)}$ and $d=\dim_{\mathbb F_2}H[2]$, then $|Q_J(G)|=2^q$. Moreover, since every homomorphism from $G$ to $H$ annihilates $G^{(2)}$ and has image in $H[2]$, restriction to the square quotient gives
\[
\mathrm{Hom}(G,H)\cong\mathrm{Hom}_{\mathbb F_2}(G/G^{(2)},H[2]).
\]
Hence
\[
\dim_{\mathbb F_2}\mathrm{Hom}(G,H)=qd,
\]
and the sharpened formula follows.
\end{proof}

\begin{proposition}[Exact even-dihedral obstruction]\label{prop:even-dihedral-obstruction}
Let $m=2k$ be even and let
\[
D_m=\langle r,s\mid r^m=e,\ s^2=e,\ srs=r^{-1}\rangle.
\]
Then the Jensen equivalence classes are
\[
\langle r^2\rangle,
\qquad r\langle r^2\rangle,
\qquad s\langle r^2\rangle,
\qquad sr\langle r^2\rangle.
\]
Consequently every $f\in S_1(D_m,H)$ is determined uniquely by three elements $u,c_0,c_1\in H[2]$ through
\[
f(r^{2t})=0,
\qquad
f(r^{2t+1})=u,
\qquad
f(sr^{2t})=c_0,
\qquad
f(sr^{2t+1})=c_1.
\]
Moreover,
\[
\mathrm{Hom}(D_m,H)=\{(u,c_0,c_1)\in H[2]^3:c_1=c_0+u\},
\]
and therefore
\[
\mathcal O(D_m,H)\cong H[2].
\]
\end{proposition}

\begin{proof}
Let
\[
\alpha:D_m\to C_2\times C_2
\]
be the homomorphism defined by $\alpha(r)=(1,0)$ and $\alpha(s)=(0,1)$. This homomorphism is well-defined because $m$ is even and $C_2\times C_2$ has exponent $2$. For every $x,y\in D_m$ we have $\alpha(y^{-1})=\alpha(y)$, and hence
\[
\alpha(xy)=\alpha(x)+\alpha(y)=\alpha(x)+\alpha(y^{-1})=\alpha(xy^{-1}).
\]
Thus each Jensen equivalence class is contained in a fibre of $\alpha$. These four fibres are exactly
\[
\langle r^2\rangle,
\qquad r\langle r^2\rangle,
\qquad s\langle r^2\rangle,
\qquad sr\langle r^2\rangle.
\]
It remains to prove that each fibre is contained in one Jensen equivalence class. For rotations, the defining relation with $x=r^a$ and $y=r^b$ gives
\[
r^{a+b}\equiv_J r^{a-b}.
\]
Taking $a=b$ shows $r^{2b}\equiv_J e$ for all $b$, so all even rotations lie in $[e]_J$. The same relation changes the exponent of a rotation by an even integer, so all odd rotations are equivalent to one another.

For reflections, use $x=sr^a$ and $y=r^b$. Then
\[
xy=sr^{a+b},
\qquad
xy^{-1}=sr^{a-b},
\]
and therefore
\[
sr^{a+b}\equiv_J sr^{a-b}.
\]
Again the exponent changes by an even integer. Hence all even-indexed reflections are equivalent to one another, and all odd-indexed reflections are equivalent to one another. This proves the asserted list of Jensen equivalence classes.

The description of $S_1(D_m,H)$ now follows from Theorem~\ref{thm:obstruction}. Since the class of $e$ is $\langle r^2\rangle$, its value is forced to be $0$, while the remaining three classes may be assigned arbitrary elements of $H[2]$.

A homomorphism $D_m\to H$ is determined by the two values $f(r)=u$ and $f(s)=c_0$. Because $r$ and $s$ are products of involutions in $D_m$, these values lie in $H[2]$. Conversely, any $u,c_0\in H[2]$ define a homomorphism: the relation $r^m=e$ is respected since $m$ is even, $s^2=e$ is respected since $2c_0=0$, and the relation $srs=r^{-1}$ is respected because in additive notation it becomes
\[
c_0+u+c_0=-u,
\]
which is equivalent to $2c_0+2u=0$. For such a homomorphism,
\[
f(sr)=f(s)+f(r)=c_0+u.
\]
Thus homomorphisms are precisely the triples with $c_1=c_0+u$. The quotient
\[
H[2]^3\big/\{(u,c_0,c_0+u):u,c_0\in H[2]\}
\]
is naturally isomorphic to $H[2]$, for example by the invariant
\[
(u,c_0,c_1)\longmapsto c_1+c_0+u.
\]
This proves $\mathcal O(D_m,H)\cong H[2]$.
\end{proof}

\section{Applications and concrete families}\label{sec:apps}

We now apply our general rigidity criterion and obstruction theory to concrete families of groups. The presentation is organized as a structural progression from discrete combinatorics to continuous geometry. We begin by isolating the square-root mechanism in symmetric groups and odd-rank Coxeter systems. We then establish a strong rigidity phenomenon for semidirect products, contrast it with the exact one-copy $H[2]$ defect found in even dihedral groups, and culminate by proving that our purely algebraic framework governs continuous Lie geometry on the orthogonal group $O(n)$ without requiring any analytical regularity.

\subsection{Symmetric groups}
We recover the known solution structure on $S_n$ in a way that isolates the square-root mechanism.

\begin{theorem}[Symmetric group]\label{thm:Sn}
Let $n\ge 2$ and let $\mathcal I$ be the set of all transpositions in $S_n$. Then $(\mathrm{SR}_2(\mathcal I))$ holds. Consequently,
\[
S_1(S_n,H)=S_{1,2}(S_n,H)=\mathrm{Hom}(S_n,H).
\]
Moreover, every $f\in S_1(S_n,H)$ is determined by a choice of $u\in H[2]$ and satisfies
\[
f(\sigma)=\begin{cases}
0,&\sigma\ \text{even},\\
u,&\sigma\ \text{odd}.
\end{cases}
\]
\end{theorem}

\begin{proof}
Let $\tau_1,\tau_2$ be transpositions. If $\tau_1=\tau_2$, then $\tau_1\tau_2=e$ is a square. If they share one element, $\tau_1\tau_2$ is a $3$-cycle, hence a square of its inverse. If they are disjoint, say $\tau_1=(ab)$ and $\tau_2=(cd)$ with $a,b,c,d$ distinct, then
\[
(acbd)^2=(ab)(cd)=\tau_1\tau_2.
\]
Thus $(\mathrm{SR}_2)$ holds, and Theorem~\ref{thm:hom} applies. The parity form follows from Corollary~\ref{cor:parity}.
\end{proof}

\subsection{Generalized dihedral groups and semidirect-product rigidity}

The next result gives a precise characterization for generalized dihedral groups, thereby addressing the scope of $(\mathrm{SR}_2)$ in a broad family beyond $S_n$ and cyclic dihedral groups.

\begin{definition}[Generalized dihedral group]
Let $A$ be an abelian group. The \emph{generalized dihedral group} of $A$ is
\[
\mathrm{Dih}(A):=A\rtimes \langle s\rangle,
\]
where $s^2=e$ and $sas=a^{-1}$ for all $a\in A$.
The coset $sA=\{sa:a\in A\}$ consists of involutions, called \emph{reflections}.
\end{definition}

\begin{theorem}[Generalized dihedral groups and surjective squaring]\label{thm:gen-dih}
Let $A$ be an abelian group, let
\[
G=\mathrm{Dih}(A)=A\rtimes \langle s\rangle,
\]
and let $\mathcal I=sA$ be the set of reflections. Then $G$ satisfies $(\mathrm{SR}_2(\mathcal I))$ if and only if the square map
\[
A\to A,
\qquad a\mapsto a^2,
\]
is surjective. Consequently, if the square map on $A$ is surjective, for instance if $A$ is finite of odd order or $A$ is divisible, then for every abelian $H$,
\[
S_1(G,H)=S_{1,2}(G,H)=\mathrm{Hom}(G,H).
\]
\end{theorem}

\begin{proof}
First assume that the square map on $A$ is surjective. The set $\mathcal I=sA$ generates $G$ because $s\in sA$ and, for each $a\in A$, one has $a=s(sa)$.

For $sa,sb\in sA$, we have
\[
(sa)(sb)=s a s b=(sas)b=a^{-1}b\in A.
\]
By surjectivity of squaring on $A$, choose $u\in A$ with
\[
u^2=a^{-1}b=(sa)(sb).
\]
Thus $(\mathrm{SR}_2(\mathcal I))$ holds.

Conversely, assume that $G$ satisfies $(\mathrm{SR}_2(\mathcal I))$. Let $c\in A$ be arbitrary. Then
\[
(s)(sc)=c.
\]
Since $s,sc\in \mathcal I$, the condition $(\mathrm{SR}_2(\mathcal I))$ yields some $t\in G$ such that
\[
t^2=c.
\]
If $c\ne e$ and $t\in sA$, then $t$ is a reflection, hence $t^2=e$, a contradiction. Therefore, for every $c\in A\setminus\{e\}$, any square root $t$ of $c$ must lie in $A$, and thus $c=t^2$ is a square in $A$. Since also $e=e^2$, the square map on $A$ is surjective.

The final conclusion follows from Theorem~\ref{thm:hom}.
\end{proof}

\begin{theorem}[Rigidity for $A\rtimes_\sigma C_2$]\label{thm:semidirect-rigidity}
Let $A$ be an abelian group, let $\sigma\in\mathrm{Aut}(A)$ satisfy $\sigma^2=\mathrm{id}_A$, and set
\[
G=A\rtimes_\sigma\langle s\rangle,
\qquad s^2=e,
\qquad sas=\sigma(a)\quad(a\in A).
\]
Suppose there exists a set $\mathcal I$ of involutions in $G$ such that $G$ satisfies $(\mathrm{SR}_2(\mathcal I))$. Then
\[
\sigma(a)=a^{-1}\qquad(\forall a\in A),
\]
and the square map $A\to A$, $a\mapsto a^2$, is surjective. Conversely, if $\sigma(a)=a^{-1}$ for all $a\in A$ and the square map on $A$ is surjective, then $G$ satisfies $(\mathrm{SR}_2(sA))$.
\end{theorem}

\begin{proof}
Let $\eta:G\to C_2$ be the quotient homomorphism with kernel $A$. Since $\eta(t^2)=0$ for every $t\in G$, every square in $G$ lies in $A$.

Because $\mathcal I$ generates $G$, it contains at least one element outside $A$. Suppose that $i\in\mathcal I\cap A$ and $j\in\mathcal I\setminus A$. Then $ij\notin A$, but $(\mathrm{SR}_2(\mathcal I))$ requires $ij$ to be a square in $G$, contradicting the preceding paragraph. Hence $\mathcal I\cap A=\varnothing$, and therefore $\mathcal I\subseteq sA$.

For $a\in A$, the element $sa$ is an involution if and only if
\[
(sa)^2=sasa=(sas)a=\sigma(a)a=e.
\]
Define the norm map
\[
N:A\to A,
\qquad N(a):=a\sigma(a),
\]
and set $K:=\ker N$. The preceding computation shows that every element of $\mathcal I$ lies in $sK$. If $a,b\in K$, then
\[
(sa)(sb)=\sigma(a)b=a^{-1}b\in K.
\]
It follows that the subgroup generated by $\mathcal I$ is contained in $K\rtimes\langle s\rangle$. Since $\langle\mathcal I\rangle=G$, we must have $A=K$. Thus $a\sigma(a)=e$ for every $a\in A$, which is exactly $\sigma(a)=a^{-1}$.

It remains to prove that the square map on $A$ is surjective. We now know that $G=\mathrm{Dih}(A)$. Let $c\in A$. Since $\mathcal I\subseteq sA$ and $\mathcal I$ generates $G$, the element $c$ can be written as a word of even length in elements of $\mathcal I$:
\[
c=j_1j_2\cdots j_{2r}.
\]
For each $\ell$, the element $j_{2\ell-1}j_{2\ell}$ is a square in $G$ by $(\mathrm{SR}_2(\mathcal I))$. In a generalized dihedral group, every square is either a square from $A$ or the identity, because reflections square to $e$. Hence each $j_{2\ell-1}j_{2\ell}$ belongs to the subgroup $A^2:=\{a^2:a\in A\}$ of $A$. Since $A$ is abelian, $A^2$ is a subgroup, and therefore
\[
c=(j_1j_2)\cdots(j_{2r-1}j_{2r})\in A^2.
\]
Thus the square map on $A$ is surjective.

The converse is precisely Theorem~\ref{thm:gen-dih}.
\end{proof}

\begin{remark}\label{rem:norm-delta}
The proof explains the role of the norm map in arbitrary semidirect products. The involutions in the coset $sA$ are exactly $sK$, where $K=\ker N$ and $N(a)=a\sigma(a)$. The difference map
\[
\delta:A\to A,
\qquad \delta(a):=a^{-1}\sigma(a),
\]
satisfies $\mathrm{im}(\delta)\subseteq K$. If $\delta$ is surjective onto $A$, then $K=A$, so $\sigma(a)=a^{-1}$ for every $a\in A$. In that case $\delta(a)=a^{-2}$, and surjectivity of $\delta$ is equivalent to surjectivity of the square map on $A$. Thus the norm and difference maps do not produce genuinely new arbitrary-$\sigma$ cases for $(\mathrm{SR}_2)$; they force the generalized dihedral situation.
\end{remark}

\begin{corollary}[Odd dihedral groups]\label{cor:odd-dihedral}
For $m$ odd, the classical dihedral group $D_m$ satisfies $(\mathrm{SR}_2)$ with reflections as involutions. Hence
\[
S_1(D_m,H)=S_{1,2}(D_m,H)=\mathrm{Hom}(D_m,H)
\]
for all abelian $H$.
\end{corollary}

\begin{proof}
This is Theorem~\ref{thm:gen-dih} with $A=C_m$, since the square map on $C_m$ is surjective exactly when $m$ is odd.
\end{proof}

\subsection{Even dihedral groups and a failure mechanism}

When $(\mathrm{SR}_2)$ fails, the homomorphism conclusion need not hold. The exact computation in Proposition~\ref{prop:even-dihedral-obstruction} shows that even dihedral groups have precisely one independent obstruction copy of $H[2]$.

\begin{theorem}[Dihedral dichotomy]\label{thm:dihedral}
Let
\[
D_m=\langle r,s\mid r^m=e,
\ s^2=e,
\ srs=r^{-1}\rangle
\]
and
\[
\mathcal I=\{sr^k:0\le k<m\}.
\]
\begin{enumerate}
\item[(i)] If $m$ is odd, then $D_m$ satisfies $(\mathrm{SR}_2(\mathcal I))$.
\item[(ii)] If $m$ is even, then $(\mathrm{SR}_2(\mathcal I))$ fails. Moreover, if $H$ is an abelian group with $H[2]\ne 0$, then there exists a normalized solution of \eqref{eq:J1} with values in $H$ that is not a group homomorphism.
\end{enumerate}
\end{theorem}

\begin{proof}
(i) This is Corollary~\ref{cor:odd-dihedral}.

(ii) Assume that $m$ is even. Then
\[
(s)(sr)=r.
\]
On the other hand, the squares in $D_m$ are exactly the even powers of $r$ together with $e$: indeed,
\[
(r^\ell)^2=r^{2\ell}
\qquad\text{and}\qquad
(sr^k)^2=e.
\]
Hence $r$ is not a square in $D_m$, so $(\mathrm{SR}_2(\mathcal I))$ fails.

Now assume $H[2]\ne 0$ and choose $u\in H[2]$ with $u\ne 0$. Proposition~\ref{prop:even-dihedral-obstruction} gives a normalized Jensen solution by taking
\[
f(r^{2t})=0,
\qquad
f(r^{2t+1})=u,
\qquad
f(sr^{2t})=0,
\qquad
f(sr^{2t+1})=0.
\]
It is not a homomorphism, since
\[
f(sr)=0
\qquad\text{but}\qquad
f(s)+f(r)=u\ne 0.
\]
\end{proof}

\subsection{Coxeter-type sufficient conditions}

The rank--$2$ lemma provides a convenient criterion in Coxeter-type settings where products of generators have odd order.

\begin{example}[Odd rank--$2$ Coxeter systems]\label{ex:coxeter-odd}
Let $(W,S)$ be a Coxeter system in the standard sense of \cite{BjornerBrenti2005,Davis2008}, so $S$ is a set of involutions and $W=\langle S\rangle$ with relations $(st)^{m_{st}}=e$. If $m_{st}$ is finite and odd for all $s\ne t$, then Corollary~\ref{cor:odd-rank2} yields $(\mathrm{SR}_2^{\mathrm{loc}}(S))$, hence every normalized Jensen solution on $W$ is a homomorphism.
\end{example}

\begin{remark}
Example~\ref{ex:coxeter-odd} is complementary to the symmetric-group application: the standard Coxeter generators of $S_n$ do not satisfy the oddness hypothesis because many $m_{st}=2$. The point is that $(\mathrm{SR}_2)$ depends on the chosen involution generating set, and $S_n$ admits a different involution generating set, namely all transpositions, for which $(\mathrm{SR}_2)$ holds.
\end{remark}

\subsection{Continuous Lie geometry: Orthogonal and unitary groups}

The final application is continuous rather than discrete. No continuity is imposed on $f$; the conclusion is purely algebraic and follows from the same square-root mechanism applied to the abstract group $O(n)$. The reflection-generation input is the Euclidean content of the Cartan--Dieudonn\'e theorem; see \cite{Gallier2011} for a standard account.

\begin{definition}[Hyperplane reflection]
Let $V$ be a finite-dimensional real inner product space. For a unit vector $u\in V$, define
\[
\tau_u(x):=x-2\langle x,u\rangle u\qquad(x\in V).
\]
Then $\tau_u\in O(V)$ is the reflection fixing the hyperplane $u^\perp$ pointwise and sending $u$ to $-u$.
\end{definition}

\begin{theorem}[Orthogonal groups]\label{thm:On}
Let $V$ be a real inner product space of dimension $n\ge 2$, and let
\[
\mathcal R:=\{\tau_u:u\in V,\ \|u\|=1\}\subseteq O(V)
\]
be the set of all hyperplane reflections. Then $O(V)$ satisfies $(\mathrm{SR}_2(\mathcal R))$. Consequently, for every abelian group $H$,
\[
S_1(O(V),H)=S_{1,2}(O(V),H)=\mathrm{Hom}(O(V),H)\cong H[2].
\]
Under this identification, a solution is determined by an element $h\in H[2]$ and is given by
\[
f(T)=
\begin{cases}
0,& \det T=1,\\
h,& \det T=-1.
\end{cases}
\]
\end{theorem}

\begin{proof}
First we prove that $\mathcal R$ generates $O(V)$, following the standard Cartan--Dieudonn\'e induction \cite{Gallier2011}. We argue by induction on $n=\dim V$. The statement is immediate for $n=1$. Let $n\ge 2$ and let $T\in O(V)$. If $T=I$, there is nothing to prove. Otherwise choose a unit vector $v$ such that $Tv\ne v$. There is a hyperplane reflection $\rho$ with $\rho(Tv)=v$: explicitly, one may take $\rho$ to be the reflection whose normal vector is proportional to $Tv-v$. Then $\rho T$ fixes $v$ and hence preserves $v^\perp$. By induction, the restriction of $\rho T$ to $v^\perp$ is a product of hyperplane reflections in $v^\perp$. Extending each of these reflections by the identity on the line $\mathbb R v$ gives hyperplane reflections of $V$. Thus $\rho T$ is a product of elements of $\mathcal R$, and therefore $T$ is also a product of elements of $\mathcal R$.

It remains to prove the square-root condition. Let $\tau_u,\tau_v\in\mathcal R$. If $\tau_u=\tau_v$, then $\tau_u\tau_v=I$, which is a square. Assume $\tau_u\ne\tau_v$. Then $u$ and $v$ are linearly independent up to sign, and the plane
\[
P:=\mathrm{span}\{u,v\}
\]
has dimension $2$. Both $\tau_u$ and $\tau_v$ preserve $P$ and act as the identity on $P^\perp$. On the oriented Euclidean plane $P$, each restriction is a line reflection, and the product of two line reflections is a rotation. Hence $(\tau_u\tau_v)|_P$ is a planar rotation. Every planar rotation has a square root, obtained by halving the angle. Let $R\in SO(P)$ satisfy
\[
R^2=(\tau_u\tau_v)|_P.
\]
Define $\widetilde R\in O(V)$ by letting $\widetilde R$ act as $R$ on $P$ and as the identity on $P^\perp$. Then
\[
\widetilde R^2=\tau_u\tau_v.
\]
Thus $(\mathrm{SR}_2(\mathcal R))$ holds.

By Theorem~\ref{thm:hom}, every normalized solution in $S_1(O(V),H)$ is a homomorphism. It remains only to identify the homomorphism group. The determinant gives a surjective homomorphism
\[
\det:O(V)\to\{\pm 1\}\cong C_2
\]
with kernel $SO(V)$. We claim that $[O(V),O(V)]=SO(V)$. Since the determinant is abelian, every commutator has determinant $1$, so $[O(V),O(V)]\subseteq SO(V)$.

Conversely, $SO(V)$ is generated by planar rotations. We prove this by induction on $n$. Fix an orthonormal basis $e_1,\dots,e_n$ and let $Q\in SO(V)$. If $Qe_1=e_1$, then $Q$ preserves $e_1^\perp$ and the induction hypothesis applies to the restriction of $Q$ to $e_1^\perp$. If $Qe_1\ne e_1$, choose a planar rotation $R$ sending $Qe_1$ to $e_1$. When $Qe_1=-e_1$, choose a unit vector $e_2\perp e_1$ and take $R$ to be the rotation by $\pi$ in $\mathrm{span}\{e_1,e_2\}$; this sends $Qe_1$ to $e_1$. In all remaining cases, take $R$ in the plane $\mathrm{span}\{Qe_1,e_1\}$ and let it be the identity on the orthogonal complement. Then $RQ$ fixes $e_1$, preserves $e_1^\perp$, and has determinant $1$ on $e_1^\perp$. Hence $RQ$ is a product of planar rotations by induction, and so is $Q$.

Now let $R_\theta$ be a rotation by angle $\theta$ in a two-dimensional plane $P$, extended by the identity on $P^\perp$. Choose a hyperplane reflection $\tau$ whose restriction to $P$ is a line reflection and whose fixed hyperplane contains $P^\perp$. Then
\[
\tau R_\phi\tau=R_{-\phi}
\]
for every angle $\phi$, and hence
\[
[\tau,R_\phi]=\tau R_\phi\tau^{-1}R_\phi^{-1}=R_{-2\phi}.
\]
Choosing $\phi=-\theta/2$ shows that $R_\theta$ is a commutator in $O(V)$. Thus every planar rotation lies in $[O(V),O(V)]$, and so $SO(V)\subseteq[O(V),O(V)]$. Therefore
\[
O(V)_{\mathrm{ab}}\cong C_2
\]
through the determinant map. Consequently
\[
\mathrm{Hom}(O(V),H)\cong\mathrm{Hom}(C_2,H)\cong H[2],
\]
and the displayed determinant formula follows.
\end{proof}

\begin{remark}[Why $U(n)$ does not fit the involution criterion]\label{rem:unitary-obstruction}
The analogous statement for $U(n)$ cannot be formulated with involutions generating all of $U(n)$. If $U\in U(n)$ and $U^2=I$, then $U$ is diagonalizable with eigenvalues only $\pm1$, and therefore $\det U\in\{\pm1\}$. Hence every finite product of involutions in $U(n)$ has determinant $\pm1$. Since $U(n)$ contains matrices with arbitrary determinant on the unit circle, $U(n)$ is not generated by its involutions. In particular, $U(n)$ cannot satisfy $(\mathrm{SR}_2(\mathcal I))$ for any involution generating set $\mathcal I$.
\end{remark}

\appendix

\section{Three-variable switching identities}\label{app:switch}

For completeness and self-containment, this appendix supplies the elementary three-variable algebraic manipulations invoked in Proposition~\ref{prop:toolbox}(d). These identities are isolated here so that the main structural progression of the square-root argument in Section~\ref{sec:criterion} remains uninterrupted.

\begin{lemma}\label{lem:switch}
Let $f:G\to H$ satisfy \eqref{eq:J1} and $f(e)=0$. Then for all $x,y,z\in G$,
\begin{align}
\label{eq:switch1}\tag{A.1.1}
f(xyz)&=2f(x)-f(xz^{-1}y^{-1}),\\
\label{eq:switch2}\tag{A.1.2}
f(xzy)&=2f(x)-f(xy^{-1}z^{-1}).
\end{align}
Consequently,
\[
f(xyz)-f(xzy)=f(xy^{-1}z^{-1})-f(xz^{-1}y^{-1}).
\]
\end{lemma}

\begin{proof}
Apply \eqref{eq:J1} with $(x,y)$ replaced by $(x,yz)$ to obtain
\[
f(xyz)+f\bigl(x(yz)^{-1}\bigr)=2f(x).
\]
Since $(yz)^{-1}=z^{-1}y^{-1}$, this gives
\[
f(xyz)+f(xz^{-1}y^{-1})=2f(x),
\]
which is \eqref{eq:switch1}.

Similarly, applying \eqref{eq:J1} with $(x,y)$ replaced by $(x,zy)$ yields
\[
f(xzy)+f\bigl(x(zy)^{-1}\bigr)=2f(x).
\]
Since $(zy)^{-1}=y^{-1}z^{-1}$, we get
\[
f(xzy)+f(xy^{-1}z^{-1})=2f(x),
\]
which is \eqref{eq:switch2}.

Subtracting \eqref{eq:switch2} from \eqref{eq:switch1}, we obtain
\[
f(xyz)-f(xzy)=f(xy^{-1}z^{-1})-f(xz^{-1}y^{-1}).
\]
The proof of the lemma is complete.
\end{proof}

\medskip

\subsection*{Funding}
None.

\subsection*{Data availability}
Data sharing is not applicable to this article as no datasets were generated or analyzed during the current study.

\subsection*{Conflict of interest}
The author declares no conflict of interest.


\begin{thebibliography}{99}

\bibitem{Aissi2024}
Y.~Aissi, D.~Zeglami, and A.~Mouzoun,
\emph{On a Pexider--Drygas functional equation on semigroups with an endomorphism},
Filomat \textbf{38} (2024), 11159--11169.
DOI: \href{https://doi.org/10.2298/FIL2431159A}{10.2298/FIL2431159A}.

\bibitem{Akkaoui2023}
A.~Akkaoui,
\emph{Jensen's functional equation on semigroups},
Acta Math.\ Hungar.\ \textbf{170} (2023), 261--268.
DOI: \href{https://doi.org/10.1007/s10474-023-01341-7}{10.1007/s10474-023-01341-7}.

\bibitem{BjornerBrenti2005}
A.~Bj\"orner and F.~Brenti,
\emph{Combinatorics of Coxeter Groups},
Graduate Texts in Mathematics, vol.~231, Springer, Berlin, 2005.
DOI: \href{https://doi.org/10.1007/3-540-27596-7}{10.1007/3-540-27596-7}.

\bibitem{Coxeter1934}
H.~S.~M.~Coxeter,
\emph{Discrete groups generated by reflections},
Ann.\ of Math.\ (2) \textbf{35} (1934), 588--621.
DOI: \href{https://doi.org/10.2307/1968753}{10.2307/1968753}.

\bibitem{Davis2008}
M.~W.~Davis,
\emph{The Geometry and Topology of Coxeter Groups},
London Mathematical Society Monographs Series, vol.~32, Princeton University Press, Princeton, NJ, 2008.
DOI: \href{https://doi.org/10.1515/9781400845941}{10.1515/9781400845941}.

\bibitem{FadliZeglamiKabbaj2016}
B.~Fadli, D.~Zeglami, and S.~Kabbaj,
\emph{A variant of Jensen's functional equation on semigroups},
Demonstr. Math.\ \textbf{49} (2016), 413--420.
DOI: \href{https://doi.org/10.1515/dema-2016-0035}{10.1515/dema-2016-0035}.

\bibitem{Friis2004}
P.~de Place Friis,
\emph{d'Alembert's and Wilson's equations on Lie groups},
Aequationes Math.\ \textbf{67} (2004), 12--25.
DOI: \href{https://doi.org/10.1007/s00010-002-2665-3}{10.1007/s00010-002-2665-3}.

\bibitem{Gallier2011}
J.~Gallier,
\emph{The Cartan--Dieudonn\'e theorem},
in \emph{Geometric Methods and Applications},
Texts in Applied Mathematics, vol.~38, Springer, New York, 2011, 231--280.
DOI: \href{https://doi.org/10.1007/978-1-4419-9961-0_8}{10.1007/978-1-4419-9961-0\_8}.

\bibitem{Householder1958}
A.~S.~Householder,
\emph{Unitary triangularization of a nonsymmetric matrix},
J.\ ACM \textbf{5} (1958), 339--342.
DOI: \href{https://doi.org/10.1145/320941.320947}{10.1145/320941.320947}.

\bibitem{Jensen1906}
J.~L.~W.~V.~Jensen,
\emph{Sur les fonctions convexes et les in\'egalit\'es entre les valeurs moyennes},
Acta Math.\ \textbf{30} (1906), 175--193.
DOI: \href{https://doi.org/10.1007/BF02418571}{10.1007/BF02418571}.

\bibitem{LeThai2011}
C.-T.~L\^e and T.-H.~Th\'ai,
\emph{Jensen's functional equation on the symmetric group $S_n$},
Aequationes Math.\ \textbf{82} (2011), 269--276.
DOI: \href{https://doi.org/10.1007/s00010-011-0089-7}{10.1007/s00010-011-0089-7}.

\bibitem{Ng1990}
C.~T.~Ng,
\emph{Jensen's functional equation on groups},
Aequationes Math.\ \textbf{39} (1990), 85--99.
DOI: \href{https://doi.org/10.1007/BF01833945}{10.1007/BF01833945}.

\bibitem{Ng1999}
C.~T.~Ng,
\emph{Jensen's functional equation on groups, II},
Aequationes Math.\ \textbf{58} (1999), 311--320.
DOI: \href{https://doi.org/10.1007/s000100050118}{10.1007/s000100050118}.

\bibitem{Ng2001}
C.~T.~Ng,
\emph{Jensen's functional equation on groups, III},
Aequationes Math.\ \textbf{62} (2001), 143--159.
DOI: \href{https://doi.org/10.1007/PL00000135}{10.1007/PL00000135}.

\bibitem{Ng2005}
C.~T.~Ng,
\emph{A Pexider--Jensen functional equation on groups},
Aequationes Math.\ \textbf{70} (2005), 131--153.
DOI: \href{https://doi.org/10.1007/s00010-005-2785-7}{10.1007/s00010-005-2785-7}.

\bibitem{Phuc2026}
\DD. V. Ph\'uc,
\emph{A Note on Torsion Confinement and Sharpness for a Reflected Jensen Functional Equation on Groups}
Aequat. Math. \textbf{100}, Article number 34 (2026).
DOI: \href{https://doi.org/10.1007/s00010-026-01279-5}{10.1007/s00010-026-01279-5}.

\bibitem{Stetkaer2003}
H.~Stetk\ae r,
\emph{On Jensen's functional equation on groups},
Aequationes Math.\ \textbf{66} (2003), 100--118.
DOI: \href{https://doi.org/10.1007/s00010-003-2679-5}{10.1007/s00010-003-2679-5}.

\bibitem{Stetkaer2013}
H.~Stetk\ae r,
\emph{Functional Equations on Groups},
World Scientific Publishing Co., Hackensack, NJ, 2013.
DOI: \href{https://doi.org/10.1142/8830}{10.1142/8830}.

\end{thebibliography}
\end{document}